\documentclass[12pt]{amsart}
\usepackage{amssymb,amscd,amsmath}
\usepackage[all]{xy}
\newtheorem{thm}{Theorem}[section]

\newtheorem{cor}[thm]{Corollary}

\newtheorem{lem}[thm]{Lemma}

\theoremstyle{remark}

\theoremstyle{definition}

\numberwithin{equation}{subsection}

\renewcommand{\bar}{\overline}

\newcommand{\cS}{\mathcal{C}}

\newcommand{\AAA}{{\sf A}}

\newcommand{\Spin}{\mathrm{Spin}}

\newcommand{\codim}{{\rm codim}\,}
\newcommand{\supp}{{\rm supp}\,}

\newcommand{\F}{{\mathbb{F}}}
\newcommand{\Q}{{\mathbb{Q}}}
\newcommand{\Z}{{\mathbb{Z}}}

\newcommand{\GL}{\mathrm{GL}}
\newcommand{\SL}{\mathrm{SL}}
\newcommand{\SU}{\mathrm{SU}}

\newcommand{\PSL}{\mathrm{PSL}}

\newcommand{\PSU}{\mathrm{PSU}}

\newcommand{\SSS}{\mathsf{S}}

\newcommand{\Sp}{\mathrm{Sp}}

\newcommand{\Irr}{\mathrm{Irr}}
\newcommand{\IBr}{\mathrm{IBr}}

\newcommand{\diag}{\mathrm{diag}}

\newcommand{\la}{\lambda}
\newcommand{\eps}{\epsilon}
\newcommand{\al}{\alpha}
\newcommand{\gam}{\gamma}
\newcommand{\om}{\omega}

\newcommand{\ulG}{\underline{G}}

\newcommand{\Alt}{{\raise 2pt\hbox{$\scriptstyle\bigwedge$}}}

\renewcommand{\mod}{\bmod \,}

\begin{document}
\title[Waring Problem]
{Waring Problem for Finite Quasisimple Groups}

\author{Michael Larsen}
\email{larsen@math.indiana.edu}
\address{Department of Mathematics\\
    Indiana University \\
    Bloomington, IN 47405\\
    U.S.A.}

\author{Aner Shalev}
\email{shalev@math.huji.ac.il}
\address{Einstein Institute of Mathematics\\
    Hebrew University \\
    Givat Ram, Jerusa\-lem 91904\\
    Israel}

\author{Pham Huu Tiep}
\email{tiep@math.arizona.edu}
\address{Department of Mathematics\\
    University of Arizona\\
    Tucson, AZ 85721\\
    U. S. A.} 

\thanks{Michael Larsen was partially supported by NSF grants DMS-0800705
and DMS-1101424.
Aner Shalev was partially supported by an ERC Advanced grant 247034. 
The first and the second named authors were partially supported 
by a Bi-National Science Foundation United States-Israel grant 2008194. 
Pham Huu Tiep was partially supported by NSF grant DMS-0901241.}

\begin{abstract}
The classical Waring problem deals with expressing every natural 
number as a sum of $g(k)$ $k$th powers. Similar problems for
finite simple groups were studied recently, and in this paper
we study them for finite quasisimple groups $G$.

We show that for a fixed group word $w \ne 1$ and large enough
$G$ we have $w(G)^3=G$, namely every element of $G$ is a product
of $3$ values of $w$. For various families of finite quasisimple
groups, including covers of alternating groups, we obtain a stronger
result, namely $w(G)^2=G$. However, in contrast with 
the case of simple groups studied in \cite{LST}, we show that $w(G)^2=G$ 
need not hold for all large $G$; moreover, if $k>2$ then $x^ky^k$ is not 
surjective on infinitely many finite quasisimple groups.

The case $k=2$ turns out to be exceptional.
Indeed, our last result shows that every element of a finite quasisimple
group is a product of two squares. This can be regarded as a non-commutative
analogue of Lagrange's four squares theorem.
\end{abstract}

\maketitle

\newpage
\section{Introduction}

The classical Waring problem deals with expressing every natural 
number as a sum of $g(k)$ $k$th powers. Analogous group-theoretic problems
were studied extensively in recent years, with emphasis on finite simple
groups, see \cite{Sh}, \cite{LSh}, \cite{LSh2}, \cite{Sh2}, \cite{NP}, 
\cite{LBST1}, \cite{LST} and the references therein.

A result proved in \cite{LST} shows that for each $k$, every element of 
a sufficiently large (non-abelian) finite simple group can be written in 
the form $x^k y^k$.  Moreover, if $k=2$ this is true with no exceptions,
namely every element of a finite simple group is a product of two
squares, as proved independently in \cite{LBST2} and in \cite{GM}. 

Can these results be extended to finite quasisimple groups?
Recall that a quasisimple group is a perfect group which is
simple modulo its center. Expressing non-identity elements
as powers, or as word values, is often harder when these elements 
are central, which makes the case of quasisimple groups more
challenging.

It is somewhat surprising that the first result mentioned above cannot
be generalized to quasisimple groups.
In particular we show in Corollary \ref{4p} below that 
{\em for any $k>2$ there are infinitely
many finite quasisimple groups $G$ such that the map $(x,y) \mapsto
x^ky^k$ from $G \times G$ to $G$ is not surjective}.
The simplest counterexample of this kind is $G= \SL_2(q)$ where
$q \equiv 3$ or $5 (\mod 8)$, and $k=4$: here the central involution
of $G$ is not a product of two $4$th powers (see Lemma \ref{sl2} below).

In contrast with these negative results, the case
$k=2$ turns out to be exceptional in a strong sense. 
We show in Theorem \ref{main5} below
that {\em every element of a finite quasisimple group is a product 
of two squares}. This general result can be regarded
as a non-commutative analogue of Lagrange's classical four squares theorem.

Powers are a particular case of {\em words}, namely elements $w$ of
free groups $F_d$ of rank $d$. Various Waring type results 
for finite simple groups generalize to any non-trivial word $w \ne 1$.
Given a word $w \in F_d$ and a group $G$ we may consider a {\em word map}
$G^d \rightarrow G$ induced by substitution, and we denote its image
by $w(G)$. For example, if $w=[x,y]$, the commutator word, then
it was shown in \cite{LBST1} that $w(G)=G$ for all finite simple
groups $G$, proving Ore's longstanding conjecture.

Word maps on finite simple groups need not be surjective
(e.g. $x^2$ is never surjective). 
However, the main theorem of \cite{LST} asserts that for any pair $w_1,w_2$ of 
fixed non-trivial words, $w_1(G)w_2(G) = G$ for all finite simple groups $G$ 
of sufficiently high order.  
This result succeeded earlier work of Shalev \cite{Sh} and Larsen-Shalev 
\cite{LSh} proving the same statement for a triple of words $w_1,w_2,w_3$.  

Theorem \ref{main1} below generalizes the earlier result to quasisimple groups:
{\em for any fixed non-trivial words $w_1, w_2, w_3$ and for all large 
enough finite quasisimple groups $G$ we have $w_1(G)w_2(G)w_3(G)=G$.}
In particular it follows that for each $k > 1$ the word map induced
by $x^k y^k z^k$ is surjective on all large enough finite quasisimple
groups.

For quasisimple groups $G$ such that $G/Z(G)$ is an alternating group
a stronger result holds. Here we show (in Theorem \ref{alt} below) that 
$w_1(G)w_2(G)=G$ provided $G$ is large enough. A similar result
(Theorems \ref{main6} and \ref{main7} below) holds for covers of odd-dimensional 
orthogonal groups and of centerless even-dimensional 
orthogonal groups of large enough {\it rank}. On the contrary, 
this result fails for all other families of finite quasisimple groups 
of Lie type with nontrivial center
(in both directions, whether we let the rank grow or we let the size of the 
definition field grow while the rank remains bounded), 
as well as for odd-dimensional odd-characteristic spin groups of 
{\it bounded} rank, cf. Theorem \ref{main4}. 

Finally, we note that while the commutator map is surjective on all finite
simple groups, this is not the case for finite quasisimple groups.
Indeed, in \cite{LBST3} the finite quasisimple groups
in which every element is a commutator are found, and the (finitely
many) exceptions are listed. 
Thus the word $x^2y^2$ behaves better on finite quasisimple
groups than the commutator word $[x,y]$ (being always surjective),
which we find rather intriguing; it is in fact the first non-primitive
word proved to be surjective on all finite quasisimple groups.
 
We are grateful to Eamonn O'Brien for computational help in proving
Theorem \ref{main5} below. We note that after this paper was completed, 
we learned
about the recent preprint \cite{BG} by Bandman and Garion, studying 
word maps of the form $x^ay^b$ on PSL$_2(q)$ and SL$_2(q)$ using the 
trace method; 
in particular they provide an independent proof of Lemma \ref{sl2} below.
\bigskip

\section{An upper bound for the width}

In this section we study the minimal $n$ such that $w(G)^n = G$,
where $w \ne 1$ is a fixed word and $G$ is a large finite
quasisimple group. This $n$ is sometimes referred to as the
{\em width} of $w$ in $G$, and we show that it is at most $3$.
In fact we prove a somewhat more general result, as follows.

\begin{thm}\label{main1}
If $w_1,w_2,w_3$ are non-trivial words, then for all finite quasisimple 
groups $G$ of sufficiently high order, $w_1(G)w_2(G)w_3(G) = G$. 
\end{thm}

\begin{proof}

We apply the classification of finite simple groups to $G/Z(G)$.  
The sporadic cases can be ignored.
For alternating cases we prove a stronger result -- see Theorem \ref{alt} 
below. That leaves the cases when $G$ is a group of Lie type.  

We provide two proofs for this case.
The first, and shorter one, is based on the so-called Gowers'
methods (see \cite{Go}) and the Nikolov-Pyber paper \cite{NP}. 
It is shown there that if $G$ is a finite group, $m(G)$ is the
minimal degree of a non-trivial irreducible character of $G$,
and $Y_1, Y_2, Y_3$ are subsets of $G$ of size exceeding 
$|G|/m(G)^{1/3}$, then we have $Y_1 Y_2 Y_3 = G$.

Now let $G$ be a finite quasisimple group such that
$T:=G/Z(G)$ is a simple group of Lie type. Fix a word $w \ne 1$
and suppose $G$ is sufficiently large.
Lower bounds on $|w(T)|$ are given in \cite{La}, \cite{LSh} and
\cite{NP}. Suppose $T = X_r(q)$ where $r$ is the Lie rank and
$q$ the field size. By \cite{La} we have, for bounded $r$, 
$|w(T)| > a|T|$ for some $a > 0$ (depending on the bound on
$r$). For $T$ classical not of type $A$ we have $|w(T)| \ge cr^{-1}|T|$
for some absolute constant $c>0$ by \cite{LSh}. And for $T$ of type
$A$ and $r$ large enough it follows from bounds in \cite{NP} that
$|w(T)| > q^{-br}|T|$ for any fixed constant $b > 1/4$.
In what follows we use this bound with  
fixed $b$ satisfying $1/4 < b < 1/3$.

We clearly have $m(G) \ge p(T)$,
where $p(T)$ is the minimal dimension of a non-trivial projective
representation of $T$, and lower bounds on $p(T)$ 
are given in \cite{LSe}. In particular it is known that
$p(T) \ge dq^r$ for some absolute constant $d>0$.

We also have $|Z(G)| \le |S(T)|$, where $S(T)$ is
the Schur multiplier of $T$. The values of $|S(T)|$ are well known,
and we have $|S(T)| \le r+1$ with finitely many exceptions.
If $r$ is large enough then $(r+1)/(dq^r)^{1/3}$ is smaller
than both $q^{-br}$ and $cr^{-1}$; and for bounded $r$ and large $q$
we have $(r+1)/(dq^r)^{1/3} < a$. 

Putting everything together we conclude that, for fixed $w \ne 1$ and 
large enough $G$ we have 
\[
|w(G)| \ge |w(T)| > |T||S(T)|/p(T)^{1/3} \ge |G|/m(G)^{1/3}.
\]
Thus, given $w_1, w_2, w_3 \ne 1$ we obtain lower bounds on
$|w_i(G)|$ as above, and these imply, using \cite{NP},
that $w_1(G)w_2(G)w_3(G) = G$ if $G$ is sufficiently large. 
This completes the first proof of Theorem 1.
\medskip 

We now provide the second proof, which does not rely on \cite{NP},
and proves a bit more, namely: {\em there exist conjugacy classes 
$C_1, C_2, C_3$ of $G$ such that $C_i \subset w_i(G)$ and
$C_1C_2C_3 = G$.} This result cannot be proved by Gowers' method
since conjugacy classes $C_i \subset G$ are not large enough.

It suffices to treat the case that $G$ is centrally closed, so this is what we do.  If $G$ is of Suzuki or Ree type, it is simple, so the result is covered by the main theorem of \cite{LST}.  We may therefore assume that $G = \ulG(\F_q)$ where
$\ulG$ is a simple, simply connected algebraic group over the field $\F_q$.

For any fixed $r_0$, it suffices to prove the theorem when the rank $r$ 
of $\ulG$ is greater than $r_0$.
Indeed, by the proof of \cite[Theorem 1.7]{LSh}, for fixed rank and $q$ sufficiently large, every non-central element of $G$ lies in $w_1(G) w_2(G)$.  For fixed rank, the order of $Z(G)$ is bounded independent of $q$, while $|w_3(G)|$ goes to infinity with $q$, for instance by \cite{La}. Hence for $z \in Z(G)$
we see that $w_1(G)w_2(G) \cap z w_3(G)^{-1} \ne \emptyset$ for bounded $r$
and large $q$, yielding $z \in w_1(G)w_2(G)w_3(G)$ as required.

Thus, we have reduced to the case where $\ulG$ is a simply connected group of 
type $A$, $B$, $C$, or $D$, of sufficiently high rank.  For cases 
$B$, $C$, and $D$, we use \cite[Theorem 1.13]{LSh} to prove that
$w_i(G)| > c r^{-1} |G|$ for some non-zero constant $c$ which depends on $w_i$ 
but not on $G$.  By Schur's method, it suffices to prove that there exist 
elements $g_i\in w_i(G)$ such that
$$\sum_{\chi\in \Irr(G)} \frac{\chi(g_1)\chi(g_2)\chi(g_3)\chi(h)}{\chi(1)^2} \neq 0$$
for all $h\in G$.
Indeed, this implies that $C_1C_2C_3 = G$, where $C_i$ is the conjugacy
class of $g_i$ in $G$ ($i=1,2,3$).

By the triangle inequality, it suffices to prove
$$\sum_{\chi\neq 1_G} \frac{|\chi(g_1)\chi(g_2)\chi(g_3)|}{\chi(1)} < 1.$$
If there exist $g_i\in w_i(G)$ such that $|\chi(g_i)| \le \chi(1)^{1/6}$ for 
all $\chi \in \Irr(G)$, then the sum above is at most 
$\sum_{\chi \ne 1_G} \chi(1)^{-1/2}$, which tends
to $0$ as $|G| \rightarrow \infty$ by \cite[Theorem 1.2]{LiSh}. 
Hence the theorem follows in this case.

Consider the set 
$$B_G := \bigcup_{\chi\in \Irr(G)}  \{g\in G:\, |\chi(g)| > \chi(1)^{1/6}\}.$$
Since $\sum_{g \in G}|\chi(g)|^2 = |G|$ for each character $\chi \in \Irr(G)$,
we have 
$$|B_G| \le \sum_{\chi\neq 1_G} |\{g\in G:\, |\chi(g)| > \chi(1)^{1/6}\}| 
  \le \sum_{\chi\neq 1_G}\frac{|G|}{\chi(1)^{1/3}}.$$
If $m(G)$ denotes the minimal degree of a non-trivial irreducible 
representation of $G$ as before, then
$$\sum_{\chi\neq 1_G} \chi(1)^{-1/3} 
  \le \sum_{\chi\neq 1_G} \chi(1)^{-1/6} \chi(1)^{-1/6}
  \le m(G)^{-1/6} \sum_{\chi\neq 1_G} \chi(1)^{-1/6}.$$
By \cite[Theorem 1.2]{LiSh} again, the sum on the right tends to $0$ for $r$ 
sufficiently large, and this implies 
$$|B_G| \le \frac{|G|}{m(G)^{1/6}}.$$
By \cite{LSe}, we have $cr^{-1} > m(G)^{-1/6}$ for large $r$.
Now, if $\ulG$ is not of type $A$ this implies
$|w_i(G)| \ge cr^{-1}|G| > |B_G|$ when $r$ is sufficiently large, 
so there are elements $g_i \in w_i(G)$ ($i=1,2,3$) such that 
$\chi(g_i) \le \chi(1)^{1/6}$ for all $\chi \in G^*$,
and this proves the theorem, except in case $A$.

For $G = \SL_{r+1}$ and $G = \SU_{r+1}$,  \cite[Proposition 6.2.2]{LST} 
and \cite[Proposition 6.2.3]{LST} respectively show that $w(G)$ contains a 
regular semi\-simple element $g\in G$ whose associated permutation is either 
a single $r$-cycle or a product of two disjoint cycles of different length and 
of total length $r$.  In either case, the centralizer of this permutation in 
$\SSS_r$ is abelian and of order $O(r^2)$, so by \cite[Corollary 6]{GLL}, 
$|\chi(g)| = O(r^3)$
for all $\chi\in \Irr(G)$.  It follows that for $r$ sufficiently large, 
$|\chi(g)| \le \chi(1)^{1/6}$
for all irreducible characters of $G$.  
We then proceed as for the $B$, $C$, and $D$ cases.
\end{proof}

\bigskip

\section{Width 2 results for some families of quasisimple groups}

In this section we improve the bound on the width for various
families of quasisimple groups, reducing it from 3 to 2.

The following theorem treats the alternating group side of the question.
Apart from its intrinsic interest it also provides a useful tool in
proving width 2 for various covers of simple groups of Lie type.

\begin{thm}\label{alt}
Let $w_1$ and $w_2$ be non-trivial words.  For all $n$ sufficiently large, 
if $G$ is a quasisimple group with $G/Z(G) \cong \AAA_n$, 
then $w_1(G)w_2(G) = G$.
\end{thm}

The proof depends on two preliminary lemmas.

\begin{lem}
\label{general}
Let $H_1$, $H_2$, and $G$ be finite groups, $\phi_i\colon H_i\to G/Z(G)$
be homomorphisms, and $x_i\in w_i(H_i)$ be elements such that
$$\phi_1(x_1)^{G/Z(G)} \phi_2(x_2)^{G/Z(G)} = G/Z(G)$$
and such that the elements of $G$ lying over $\phi_1(x_1)$ lie in a single conjugacy class of $G$.  Then $w_1(G)w_2(G) = G$.
\end{lem}

\begin{proof}
Let $\tilde g\in G$ map to $g\in G/Z(G)$.  There exist elements $y_i\in w_i(G/Z(G))$
such that $y_i$ is conjugate to $x_i$ and $y_1 y_2 = g$.  Therefore there exist
elements $\tilde y_i \in w_i(G)$ mapping to $y_i$ such that $\tilde y_1 \tilde y_2 = \tilde g \tilde z$
for some element $\tilde z\in Z(G)$.  However, $\tilde z^{-1} \tilde y_1$ is conjugate to $\tilde y_1$
and therefore also lies in $w_1(G)$, so $g\in w_1(G) w_2(G)$.
\end{proof}

\begin{lem}
\label{special}
Suppose that there exist finite groups $I_k$ and $J_k$ for $k=1,\,2$
and integers $d,e\ge 5$ with the following properties:
\begin{itemize}
\item For $1\le k\le 2$, $w_k(I_k)$ contains an element $y_k$ of order $d$;
\item For $1\le k\le 2$, $w_k(J_k)$ contains an element $z_k$ of order $e$;
\item $|I_1|\cdot |I_2|$ and $|J_1|\cdot |J_2|$ are relatively prime;
\item $d$ is even.
\end{itemize}
Then for all $n$ sufficiently large, every quasisimple group $G$ with $G/Z(G)\cong \AAA_n$
satisfies $w_1(G)w_2(G) = G$.
\end{lem}

\begin{proof}
Assuming $n\ge 8$, either $G\cong \AAA_n$ or $G \cong \tilde{\AAA}_n$ 
is a double cover of $\AAA_n$.
Let $H_1 = I_1\times J_1$ and $H_2 = I_2\times J_2$.
By hypothesis, $|J_1|$ and $|J_2|$ are odd and prime to $|I_1|$ and $|I_2|$.  Therefore,
for $n$ sufficiently large, there exist non-negative integers $a_1,a_2,b_1,b_2$ such that
$a_1$ and $a_2$ are even, 
$$a_1 |I_1|  = a_2 |I_2|,\ b_1|J_1| = b_2 |J_2|,$$
and
$$a_1 |I_1| + b_1|J_1| = n.$$
We define $\phi_i\colon I_i\times J_i\to \AAA_n$ as follows.  The regular representation
embeds $I_k$ in $\SSS_{|I_k|}$ and $J_k$ in $\SSS_{|J_k|}$; the latter has image in $\AAA_{|J_k|}$.  Composing with the diagonal embeddings $\SSS_{|I_k|}\to \SSS_{|I_k|}^{a_k}$ and $\SSS_{|J_k|}\to \SSS_{|J_k|}^{b_k}$ and the embedding 
$$\SSS_{|I_k|}^{a_k} \times \SSS_{|J_k|}^{b_k} \to \SSS_{a_k |I_k| + b_k |J_k|} = \SSS_n,$$
we have image in $\AAA_n$, since $a_k$ is even.
We let $x_k = (y_k,z_k)\in H_k$ and apply Lemma~\ref{general}.
As $\phi_1(x_1)$ has at least $a_1\ge 2$ cycles of even length $d$, by 
\cite[\S9]{Sch},
the inverse image of $\phi_1(x_1)$ in $G = \tilde \AAA_n$ lies in a single conjugacy class.
As $\phi_1(x_1)$ and $\phi_2(x_2)$ each consist of $a_1|I_1|/d$ $d$-cycles and $b_1 |J_1|/e$ $e$-cycles,
if $n$ is sufficiently large, the product of their conjugacy classes in $\AAA_n$ is all of $\AAA_n$ by
\cite[Theorem 1.10 (4)]{LSh2}.
\end{proof}

We can now prove Theorem \ref{alt}.

\begin{proof}
It suffices to show that the hypotheses of Lemma~\ref{special} can always be satisfied.
For every prime $p$, let $U_p^n = \ker(\SL_2(\Z_p)\to \SL_2(\Z/p^n\Z))$.
For all $n\ge 3$, $x\in U_p^n\setminus U_p^{n+1}$ implies
$x^p \in U_p^{n+1}\setminus U_p^{n+2}$.
By the Baire category theorem, there exists an injective map from any free group to
$U_p^3$, so in particular, for any non-trivial word $w$, $w(U_p^3)$ contains some element $\alpha\in U_p^3\setminus \{1\}$.  Let $r\ge 3$ denote the largest integer for
which $\alpha\in U_p^r$.  Then the image of $\alpha$ in $U_p^3/U_p^{r+l}$
has order $p^l$.  Thus, for every $p^l$, the finite $p$-group $U_p^3/U_p^{r+l}$
has an element of order $p^l$ in the image of the word map $w$.  Applying this for $p=2$ and $p=5$, we can satisfy the hypotheses of Lemma \ref{special} with $d=8$, $e=5$.
\end{proof}

It is intriguing that the proof of Theorem \ref{alt} on covers of alternating
groups uses $p$-adic groups (as well as new character bounds for
symmetric groups via \cite{LSh2}).

\medskip
Next we prove width two for covers of odd-dimensional orthogonal groups of
large enough rank.

\begin{lem}\label{an-spin}
Let the prime $p > 2$ be coprime to $n \geq 3$ and let $q$ be a power of $p$. Then 
the double cover $H := \tilde{\AAA}_{2n}$ of the alternating group $\AAA_{2n}$ 
embeds in $G := \Spin_{2n-1}(q)$, in such a way that the central involution 
$z$ of $H$ becomes central in $G$. 
\end{lem}

\begin{proof}
Consider the natural action of $\SSS_{2n}$ on the 
$2n$-dimensional space $W = \langle e_i \mid 1 \leq i \leq 2n\rangle_{\F_q}$,
which fixes the non-degenerate quadratic form $Q$ defined via 
$Q(e_i) = 1/2$ and $e_i \perp e_j$ whenever $i \neq j$. As an $\SSS_{2n}$-module,
$W$ splits into the direct sum $V \oplus I$, where $I := \langle \sum_ie_i \rangle$
and $V := I^{\perp}$. Then the transposition $(1,2)$ acts on $V$ as 
the reflection $\rho_u$ where $u := e_1-e_2$, and similarly the transposition 
$(3,4)$ acts as the reflection $\rho_v$ where $v := e_3-e_4$.
As the action of $\AAA_{2n}$ on $V$ is faithful, we may identify $\AAA_{2n}$ 
with its image in $\Omega(V)$. 

Since $Q(u) = Q(v) = 1$, $u^2 = v^2 = e$, the identity element in 
the Clifford algebra $C(V)$. In fact, $u$ and $v$ both belong to 
the Clifford group $\Gamma(V)$. The conjugation action on $V$ induces a 
surjective homomorphism $\varphi~:~\Gamma(V) \to SO(V)$, with 
$$\varphi(u) = -\rho_u,~~~\varphi(v) = -\rho_v,~~~\varphi(uv) = \rho_u\rho_v = 
  (12)(34).$$
Moreover, $uv \cdot uv = -u^2v^2 = -e$ and $vu \cdot uv = e$. In particular, 
$uv$ is an element of order $4$  in 
$G := \Spin(V) \cong \Spin_{2n-1}(q)$ (see e.g. \cite[\S6]{TZ} for basic
facts about spin groups). Recall that $\varphi$ projects $G$ onto $\Omega(V)$
with kernel $\langle -e \rangle \cong C_2$. Taking $H := \varphi^{-1}(\AAA_{2n})$,
we see that $H/\langle -e \rangle \cong \AAA_{2n}$, and $H$ contains the 
inverse image $uv$ of order $4$ of $(1,2)(3,4)$. It follows that 
$H \cong \tilde{\AAA}_{2n}$.    
\end{proof}

The following statements are analogues of Lemma \ref{an-spin} for symplectic and
orthogonal groups:
 
\begin{lem}\label{an-sp}
Let $p > 2$ be any prime and let $q$ be a power of $p$. 

{\rm (i)} Suppose $n \geq 9$ is an odd square. Then 
$H := \tilde{\AAA}_{n}$ embeds in $G := \Omega^{+}_{2^{(n-3)/2}}(q)$, in such a way 
that the central involution $z$ of $H$ becomes central in $G$.

{\rm (ii)} Suppose $n \equiv 4,6 (\mod 8)$, and $p {\!\not{|}} n$. Then 
$H := \tilde{\AAA}_{n}$ embeds in $G := \Sp_{2^{(n-2)/2}}(q)$, in such a way 
that the central involution $z$ of $H$ becomes central in $G$. 
\end{lem}

\begin{proof}
It is well known that $H$ has faithful irreducible representations of 
degree $2^{\lfloor n/2 \rfloor -1}$ (the so-called {\it basic spin representations}),
one if $2|n$ and two isomorphism classes if $n$ is odd. Let $\al_n$ denote the 
character of any of these representations. In the case of (i), it was shown in
\cite{Gow} that there is an $H$-invariant unimodular lattice $\Lambda$ affording
the $H$-character $\al_n$ and moreover $V := \Lambda/p\Lambda$ is irreducible 
over $H$. The action of $H$ on $\Lambda/p\Lambda$ gives rise to 
an irreducible embedding 
$$H \hookrightarrow \Omega(V) \cong \Omega^{\eps}_{2^{(n-3)/2}}(p) 
  \leq \Omega^{\eps}_{2^{(n-3)/2}}(q),$$
and moreover the type $\eps$ is $+$ since $\Omega^{\eps}_{2^{(n-3)/2}}(p)$ contains the
central involution $-1_V$ (cf. \cite[Proposition 2.5.13]{KL}).

In the case of (ii), we also have $\Q(\al_n) = \Q$ and $\al_n (\mod p)$ is 
irreducible but of type $-$, cf. the proof of \cite[Theorem 1.2]{T}. It follows
that $H$ acts faithfully and irreducibly on a non-degenerate symplectic 
space $V = \F_q^{2^{(n-2)/2}}$.   
\end{proof}

Now we prove the following extension of \cite[Proposition 6.3.1]{LST}:

\begin{cor}\label{spin1}
Let $w_1$ and $w_2$ be non-trivial words and $k\ge 3$ an integer.
Then there exists $N$ such that for all $l > N$ and all $q$,
$$w_1(\Spin_{2kl+1}(q)) w_2(\Spin_{2kl+1}(q)) = \Spin_{2kl+1}(q).$$
\end{cor}

\begin{proof}
Let $L := \Spin_{2kl+1}(q)$. The proof of \cite[Proposition 6.3.1]{LST} actually
shows that there exists $N$ such that $w_1(L)w_2(L)$ contains any 
non-central element of $L$, for all $l > N$ and for all $q$.
In particular, we are done if $q$ is even. Suppose $q = p^a$ and $p > 2$ is 
a prime. Since $Z(L) = \langle z \rangle \cong C_2$, we need to show
that $z \in w_1(L)w_2(L)$.  

By Theorem \ref{alt}, there is $N_1$ such that 
$w_1(\tilde{\AAA}_n)w_2(\tilde{\AAA}_n) = \tilde{\AAA}_n$ for all $n \geq N_1$.
Set $N_2 := N_1$ if $p{\!\not{|}}N_1$ and $N_2 := N_1+1$ if $p|N_1$. By Lemma 
\ref{an-spin} there is an embedding 
$H := \tilde{\AAA}_{2N_2} \hookrightarrow \Spin_{2N_2-1}(q)$ sending the central 
involution $z_1$ of $H$ to the central involution $z_2$ of 
$\Spin_{2N_2-1}(q)$.    

Replacing $N$ by $\max(N,N_2)$ we may assume $N \geq N_2$. Then there is 
an embedding $\Spin_{2N_2-1}(q) \hookrightarrow L$ sending the 
central involution $z_2$ of $\Spin_{2N_2-1}(q)$ to $z$ 
(this can be seen by applying \cite[Lemma 4.1]{LBST3}). Thus 
we can embed $H = \tilde{\AAA}_{2N_2}$ into $L$ and identify 
$z_1$ with $z$. It follows that $z \in w_1(H)w_2(H) \subseteq w_1(L)w_2(L)$, 
and so we are done.  
\end{proof}

Theorem \ref{alt} and Lemma \ref{an-sp}(i) yield the following extension of 
\cite[Proposition 6.3.2]{LST}:

\begin{cor}\label{spin2}
If $w = w_1w_2$, where $w_1$ and $w_2$ are non-trivial disjoint words,
there exists an integer $N$ such that if $2n$ is divisible by $2^N$,
for every odd prime power $q$, $w(\Spin^+_{2n}(q))$ contains 
$Z(\Spin^{+}_{2n}(q))$.
\end{cor}

\begin{proof}
By Theorem \ref{alt}, there is some $N_1$ depending on $w$ such that 
$w(\tilde{\AAA}_m)$ contains the central involution of $\tilde{\AAA}_m$ for 
all $m \geq N_1$. Now choose $m \geq \max\{N_1,9\}$ to be an odd square. 
By Lemma \ref{an-sp}(i), $w(\Omega^+_{2^{(m-3)/2}}(q))$ contains the central
involution of $\Omega^+_{2^{(m-3)/2}}(q)$. Now for any $n$ divisible by 
$2^{(m-3)/2}$, by embedding $\Omega^+_{2^{(m-3)/2}}(q)$ diagonally
into $\Omega^+_{2n}(q)$, we see that $w(\Omega^+_{2n}(q))$ contains 
$-1_V$.

Let $V = \F_q^{2n}$ be the natural module for $\Omega^+_{2n}(q)$, with an 
invariant quadratic form $Q$. The conjugation action of the Clifford 
group $\Gamma(V)$ on $V$ induces a surjective homomorphism 
$\varphi~:~\Gamma(V) \to GO(V)$. Also, 
$Z(\Spin(V)) = \{e,-e,t,-t\} \cong C_2^2$, where $e$ is the identity element
in the Clifford algebra $C(V)$, and $\varphi(t) = -1_V$. Let $u \in V$ be 
any non-singular vector. Then $u \in \Gamma(V)$ and 
$$t^{-1}utu^{-1} = \varphi(t^{-1})(u) \cdot u^{-1} = -u \cdot u^{-1} = -e,$$
i.e. $utu^{-1} = -t$. Also observe that $G := \Spin(V)$ is normal 
in $\Gamma(V)$. Since $w(\Omega^+_{2n}(q))$ contains $-1_V$ and 
$\Omega^+_{2n}(q) = G/\langle -e \rangle$, we may assume that 
$t \in w(G)$. Conjugating by $u$, we also get $-t \in w(G)$. 
Certainly $e \in w(G)$.  On the other hand, the proof of Corollary 
\ref{spin1} shows that $w(\Spin_{2n-1}(q))$ contains the central involution of
$\Spin_{2n-1}(q)$ when $n$ is large enough. Embedding $\Spin_{2n-1}(q)$ 
in $G$, we get that $-e \in w(G)$.
\end{proof}

\begin{thm}\label{main6}
Let $w_1$ and $w_2$ be non-trivial words.  For all $n$ sufficiently large and for 
all $q$, if $G$ is a quasisimple group with $G/Z(G) \cong \Omega_{2n+1}(q)$, then 
$w_1(G)w_2(G) = G$.
\end{thm}

\begin{proof}
1) By the main result of \cite{LST} we may assume that $q$ is odd and furthermore
$G = \Spin_{2n+1}(q)$. We proceed along the lines of the proof of 
\cite[Proposition 6.3.5]{LST}. In particular, this proof shows that there is 
some $B > 0$ depending only on $w_1$ and $w_2$ such that 
$w_1(G)w_2(G)$ contains all elements $g \in G$ with support 
$\supp(g) > B$. Recall that the {\it support} $\supp(g)$ is defined in 
\cite{LST} to be the codimension of the largest eigenspace of $\varphi(g)$ acting 
on the $\bar{\F}_q$-space $V \otimes_{\F_q} \bar{\F}_q$, where $V = \F_q^{2n+1}$ 
is the natural module for $G$, and $\varphi~:~G \to \Omega(V)$ is the 
natural projection. Also, $\ker(\varphi) = Z(G) \cong C_2$.  

It remains to show that every element $g \in G$ of support $\leq B$ lies in
$w(G)$, where we define $w(X) := w_1(X)w_2(X)$ for any group $X$. 
Without any loss we may assume $n > 2B$. By 
\cite[Proposition 4.1.2]{LST}, there is a unique eigenvalue 
$\la$ of $\varphi(g)$ such that $\codim \ker(\varphi(g)-\la) = \supp(g)$ and 
$\la = \pm 1$. 

\smallskip
2) First we consider the case $\la = 1$.
By Corollary~\ref{spin1}, there 
exists $N \geq B$ (depending on $w_1,w_2$) such that 
$w(\Spin_{6l+1}(q)) = \Spin_{6l+1}(q)$ for all $l \geq N$.
Now assume that $n \geq 4N$. 
Then it is shown in part 3a) of the proof of 
\cite[Proposition 6.3.5]{LST} that $g$ preserves
an orthogonal decomposition $V = W \oplus U$ with
$W$ of dimension $6N+1$ and $\varphi(g)|_{U} = 1_{U}$.
By \cite[Lemma 4.1]{LBST3}, $G$ contains a subgroup 
$H \cong \Spin(W) = \Spin_{6N+1}(q)$ 
such that $\varphi$ projects $H$ onto the subgroup
$$\{ f \in \Omega(V) \mid f_U = 1_U \} \cong \Omega(W)$$
with kernel $Z(G)$. Now $\varphi(g) \in \varphi(H)$ and 
$H > Z(G) = \ker(\varphi)$. It follows that 
$$g \in H = \Spin_{6N+1}(q) = w(\Spin_{6N+1}(q)) \subseteq w(G).$$

\smallskip
3) Finally, assume that $\la = -1$. By \cite[Proposition 6.3.2]{LST}, 
there exists an even integer $M$ (depending on $w_1,w_2$) such that, for any
$m \geq 1$ and any odd $q$, $w(\Spin^{+}_{2mM}(q))$ contains an element lying 
above the central involution of $\Omega^{+}_{2mM}(q)$. Fix an
integer $k \geq 3$ coprime to $2M$.
By Corollary~\ref{spin1}, there 
exists $N \geq B$ (depending on $w_1,w_2$) such that 
$w(\Spin_{2kl+1}(q)) = \Spin_{2kl+1}(q)$ for all $l \geq N$. 
Now assume that $n > k(N+M)$. By \cite[Lemma 6.3.3]{LST}, there are some integers
$x > N$ and $y > 0$ such that $n = xk+yM$. Then it is shown in part 3b) of
\cite[Proposition 6.3.5]{LST} that 
$g$ preserves an orthogonal decomposition $V = W \oplus U$, where
$\dim W = 2yM$, $W$ is of type $+$, $\varphi(g)|_W = -1_W$, and
$\dim U = 2xk+1$ with $x > N$. 

By the choice of $M$ and since $\Spin(W) \cong \Spin^{+}_{2yM}(q)$, there is 
some $h \in w(\Spin(W))$ lying above $-1_W$. Embedding $K = \Spin(W)$ into 
$G = \Spin(V)$ using \cite[Lemma 4.1]{LBST3}, we may assume that 
$\varphi(h) = \diag(-1_W,1_U)$. On the other hand,
$\varphi(g) = \diag(-1_W,u)$ for some $u \in \Omega(U)$;
in particular, $\varphi(h^{-1}g) = \diag(1_W,u)$.   
Again by \cite[Lemma 4.1]{LBST3}, $G$ contains a subgroup 
$L \cong \Spin(U) = \Spin_{2xk+1}(q)$
such that $\varphi$ projects $L$ onto the subgroup
$$\{ f \in \Omega(V) \mid f_W = 1_W \} \cong \Omega(U)$$
with kernel $Z(G)$. Now $\varphi(h^{-1}g) \in \varphi(L)$ and 
$L > Z(G) = \ker(\varphi)$. It follows that $g = ht$ for some 
$t \in L$. By the choice of $N$ and since 
$L \cong \Spin_{2xk+1}(q)$, $t \in w(L)$. Furthermore, $[K,L] = 1$
by \cite[Lemma 6.1]{TZ}. Consequently, 
$$\begin{array}{ll}g  = ht & \in w_1(K)w_2(K) \cdot w_1(L)w_2(L) \\ 
  & \subseteq w_1(K)w_1(L) \cdot w_2(K)w_2(L)\\ 
  & \subseteq w_1(G)w_2(G).\end{array}$$
\end{proof}

Note that Theorem \ref{main6} is false in the case $n$ is bounded (but $q$ grows),
cf. Theorem \ref{main4}(iv).

\begin{thm}\label{main7}
Let $w = w_1w_2$ be a product of two non-trivial disjoint words $w_1$ and $w_2$. 
Then there is some $N$ depending only on $w_1$ and $w_2$ such that if $n > N$, 
$\eps = \pm$, and moreover $\eps \neq (-1)^{n(q-1)/2}$ if $q$ is odd, then 
$w(\Spin^{\eps}_{2n}(q)) = \Spin^{\eps}_{2n}(q)$.
\end{thm}

\begin{proof}
By the main result of \cite{LST} we may assume that $q$ is odd. 
We proceed along the lines of the proof of \cite[Proposition 6.3.7]{LST}. 
In particular, this proof shows that there is 
some $B > 0$ depending only on $w_1$ and $w_2$ such that 
$w(G)$ contains all elements $g \in G := \Spin^{\eps}_{2n}(q)$ with support 
$\supp(g) > B$. 

Let $\varphi$ be the natural projection $G \to \Omega^{\eps}_{2n}(q)$.
Consider any element $g \in G$ with $\supp(g) \leq B$ and let $\la = \pm 1$ be the 
primary eigenvalue of $g$. If $\la = 1$, then arguing as in part 2) 
of the proof
of Theorem \ref{main6} we obtain $g \in w(G)$. So we may assume that 
$\la = -1$. Now the condition $\eps \neq (-1)^{n(q-1)/2}$ implies that 
$n$ is odd if $\eps = +$, and furthermore $\varphi(g) \neq -1_V$ 
(as $Z(\Omega^{\eps}_{2n}(q)) = 1$ in this case by \cite[Proposition 2.5.13]{KL}).
By Corollary \ref{spin2}, there exists a $2$-power $M$
(depending on $w_1,w_2$) such that 
$w(\Spin^{+}_{2mM}(q)) \supseteq Z(\Spin^{+}_{2mM}(q))$ for any $m \geq 1$. Fix
coprime odd integers $k,l \geq 3$ and an integer $v > 0$ such
that $l|(kv-1)$ and $2|(n-v)$. Then by \cite[Proposition 6.6.6]{LST}, there
exists $L \geq B$ (depending on $w_1,w_2$) such that
$$w(\Spin^{\eps}_{2s}(q)) \supseteq 
  \Spin^{\eps}_{2s}(q) \setminus Z(\Spin^{\eps}_{2s}(q))$$
for all $s = k(2al+v)$ and $a \geq L$.

Now assume that $n > kl(2L+M)+kv$. 
As in the proof of \cite[Proposition 6.3.7]{LST}, we see that
$g$ preserves an orthogonal decomposition $V = W \oplus U$ of 
the natural $G$-module $V = \F_q^{2n}$, such that
$\dim W = 2yM$, $y \geq 1$, $W$ is of type $+$, $\varphi(g)|_{W} = -1_{W}$, 
$\dim U = k(2xl+v)$ with $x > L$, and $\varphi(g)|_U$ has at least two 
eigenvalues $-1$. Since $\varphi(g) \neq -1_V$, we conclude that 
$\varphi(g)_U$ is not scalar. Hence by \cite[Proposition 6.3.6]{LST},
there is some $t \in w(\Spin(U))$ lying above $\varphi(g)|_U$.    
Embedding $K = \Spin(U)$ into 
$G = \Spin(V)$ using \cite[Lemma 4.1]{LBST3}, we may assume that 
$\varphi(t) = \diag(1_W,\varphi(g)|_U)$ and so
$\varphi(gt^{-1}) = \diag(-1_W,1_U)$.   
Again by \cite[Lemma 4.1]{LBST3}, $G$ contains a subgroup 
$L \cong \Spin(W) = \Spin^{+}_{2yM}(q)$
such that $\varphi$ projects $L$ onto the subgroup
$$\{ f \in \Omega(V) \mid f_U = 1_U \} \cong \Omega(W)$$
with kernel $Z(G)$. Now $\varphi(gt^{-1}) \in \varphi(L)$ and 
$L > Z(G) = \ker(\varphi)$. It follows that $g = ht$ for some 
$h \in L$. By the choice of $M$ and since 
$L \cong \Spin_{2xk+1}(q)$, $h \in w(L)$. Since $[K,L] = 1$
by \cite[Lemma 6.1]{TZ}, we conclude that $g \in w(G)$.
\end{proof}

Note that Theorem \ref{main7} is false in the case
$\eps = (-1)^{n(q-1)/2}$, see Theorem \ref{main4}(iii).

\begin{thm}\label{main8}
Let $w = w_1w_2$ be a product of two non-trivial disjoint words $w_1$ and $w_2$. 
Then there is some $N$ depending only on $w_1$ and $w_2$ such that if $n$ 
is divisible by $2^N$, then 
$w(\Sp_{2n}(q)) = \Sp_{2n}(q)$.
\end{thm}

\begin{proof}
By the results of \cite{LST}, it suffices to consider the case $q = p^f$ is odd. 
Next, the proof of \cite[Proposition 6.1.1]{LST} shows that $w(G)$ contains all
non-central elements of $G := \Sp_{2n}(q)$ if $n$ is large enough. 
Again by Theorem \ref{alt}, there is some $N_1$ depending on $w$ such that 
$w(\tilde{\AAA}_m)$ contains the central involution of $\tilde{\AAA}_m$
if $m \geq N_1$. Without loss we may assume $4|N_1$. Let $N_2 := N_1$ if
$p {\!\not{|}} N_1$ and $N_2 := N_1 +2$ if $p|N_1$.  Then 
$p$ is coprime to $N_2 \equiv 4,6 (\mod 8)$. Hence by Lemma \ref{an-sp}(ii),
$\tilde{\AAA}_{N_2}$ can be embedded irreducibly in $\Sp_{2^{(N_2-2)/2}}(q)$. 
Embedding the latter diagonally into $G = \Sp_{2n}(q)$, we see that 
$Z(G) \subseteq w(G)$ if $n$ is divisible by $2^{N_2/2-2}$.  
\end{proof}

Again, without the $2$-divisibility condition Theorem \ref{main7} is false,
cf. Theorem \ref{main4}(ii).
\bigskip

\section{Lower bounds for the width}

In this section we will prove several results to show that Theorem \ref{main1}
is best possible, namely that in general the width cannot be reduced to 2. 
By the main result of \cite{LST}, 
it is natural to expect 
the (non-trivial) central elements of finite quasisimple groups $G$ to be 
the main obstructions for words to have width $2$ on $G$.

In what follows, we will use the notation $\SL^{\eps}$ to denote $\SL$ when 
$\eps = +1$ and $\SU$ when $\eps = -1$, and similarly for $\GL^{\eps}$. 
Furthermore, $E^{\eps}_6(q)$ denotes $E_6(q)$ for $\eps = +1$ and 
${\!^2}E_6(q)$ for $\eps = -1$. Also, if
$n$ is an integer and $p$ is a prime, then $n_p$ denotes the $p$-part of $n$.

We start with the $\SL_2$-case:

\begin{lem}\label{sl2}
Let $q \geq 5$ be an odd prime power and let $2^{a+1} = (q^2-1)_2$ 
(so $a \geq 2$). Then the word $x^{2^a}y^{2^a}$ is not surjective on 
$G = \SL_2(q)$.
\end{lem}

\begin{proof}
1) First we observe that any (semisimple) $2$-element in $G$ has order dividing 
$q \pm 1$, whence its order is $2^b$ with $0 \leq b \leq a$. It follows 
that the central involution $z$ of $G$ cannot be any $2^a$-power.

2) Next we claim that $C_G(g^{2^a}) = C_G(g)$ for any $g \in G$ with 
$g^{2^a} \neq 1$. Indeed, since $C_G(g) \leq C_G(g^{2^a})$, it suffices to check 
that these two centralizers have the same order. Now if $\pm g$ is a nontrivial
unipotent element, then the two centralizers have order $2q$. Otherwise we 
may assume that $g$ is conjugate (over $\bar{\F}_q$) to 
$\diag(\al,\al^{-1})$ with $\al^{q - \eps} = 1$ for some $\eps = \pm 1$. 
By our hypothesis, $\al^{2^a} \neq 1$; in particular, $\al \neq \pm 1$ and so
$|C_G(g)| = q -\eps$. Furthermore, if $\al^{2^a} = -1$, then $g^{2^a} = z$, 
contradicting the observation in 1). Hence $\al^{2^a} \neq \pm 1$, and so
$|C_G(g^{2^a})| = q - \eps$.

3) Now assume that $z = x^{2^a}y^{2^a}$ for some $x, y \in G$. By the result of
1), $y^{2^a} \neq 1$, and so $C_G(y) = C_G(y^{2^a})$ by the result of 2). 
Since $y^{2^a} = x^{-2^a}z$, we have $[x,y^{2^a}] = 1$, whence 
$x \in C_G(y^{2^a}) = C_G(y)$. It follows that $z = (xy)^{2^a}$, contrary to 1).   
\end{proof}

\begin{thm}\label{main2}
Let $G = \SL_n^{\eps}(q)$, where $q$ is a prime power and $\eps = \pm 1$, and let
$p$ be a prime divisor of $\gcd(n,q-\eps)$. 

{\rm (i)} Let $a$ be defined as follows:
$$a = \left\{ \begin{array}{ll} 
  \lfloor \log_pn \rfloor + 1+ \log_p\left( \frac{q-\eps}{\gcd(n,q-\eps)} \right)_p,
    & p > 2 \\
  \lfloor \log_2n \rfloor + \log_2\left( \frac{q^2-1}{\gcd(n,q-\eps)} \right)_2,
    & p = 2. \end{array} \right.$$
Then $x^{p^a}y^{p^a}$ is not surjective on $G$. 

{\rm (ii)} Let $a$ be defined as follows:
$$a = \left\{ \begin{array}{ll} 
  \lfloor \log_pn \rfloor + \log_p(q-\eps)_p,
    & p > 2 \\
  \lfloor \log_2n \rfloor -1 + \log_2(q^2-1)_2,
    & p = 2 \end{array} \right.$$
and let $z \in G$ be a central element of order $p$. Then
$z \neq x^{p^a}y^{p^a}$ for all $x,y \in \GL^{\eps}_n(q)$. 

{\rm (iii)} If $n = p > 2$ and 
$(q - \eps)_p = p$, then $x^py^p$ is not surjective on $G$. 
\end{thm} 

\begin{proof}
1) Write $p^b = (\gcd(n,q-\eps))_p$ so that $b \geq 1$. For (i), let $z = \om I_n$ 
be a central element of order $p^b$ of $G$, $\om \in \bar{\F}_q$ and $|\om| = p^b$.
For (ii) and (iii), let $z = \om I_n$ be a central element of order $p$ of $G$, 
$\om \in \bar{\F}_q$ and $|\om| = p$.
Assume the contrary: $z = x^{p^a}y^{p^a}$ for some $x,y \in \GL^{\eps}_n(q)$ 
in (i) or (ii), or $a=1$ and $z = x^py^p$ for some $x, y \in G$ in (iii). 
Let $x = su$ and $y = tv$ be the Jordan decompositions of $x$ and $y$, where 
$s$ and $t$ are semisimple, and $u$ and $v$ are unipotent.

\smallskip
2) Suppose $\la_1, \ldots ,\la_n$ are all the eigenvalues of the matrix $s$
and $\mu_1, \ldots ,\mu_n$ are all the eigenvalues of the matrix $t$
(with counting multiplicities). Then $\la_1^{p^a}, \ldots ,\la_n^{p^a}$ are all 
the eigenvalues of $s^{p^a}$ and $\mu_1^{p^a}, \ldots ,\mu_n^{p^a}$ are all 
the eigenvalues of $t^{p^a}$.
By our assumption, $[x^{p^a},y^{p^a}] = 1$. It follows that the semisimple 
elements $s^{p^a}$ and $t^{p^a}$ also commute, and so (over $\bar{\F}_q$) we 
can simultaneously conjugate them to diagonal matrices. So without loss 
(and relabeling the $\la_i$ and $\mu_j$ suitably) we may assume that 
$$s^{p^a} = \diag\left(\la_1^{p^a}, \ldots ,\la_n^{p^a}\right),~~~
  t^{p^a} = \diag\left(\mu_1^{p^a}, \ldots ,\mu_n^{p^a}\right).$$
In particular,
$$s^{p^a}t^{p^a} = \diag\left((\la_1\mu_1)^{p^a}, \ldots ,(\la_n\mu_n)^{p^a}
  \right).$$
Since $s^{p^a}$ and $u^{p^a}$ are powers of $x^{p^a}$, and 
$t^{p^a}$ and $v^{p^a}$ are powers of $y^{p^a}$, they all commute with each other.
But $s^{p^a}$ and $t^{p^a}$ are semisimple, so $s^{p^a}t^{p^a}$ is semisimple.
Similarly, since $u^{p^a}$ and $v^{p^a}$ are unipotent, $u^{p^a}v^{p^a}$ is 
unipotent. Furthermore, $s^{p^a}t^{p^a}$ and $u^{p^a}v^{p^a}$ commute.
Now 
$$z = (su)^{p^a}(tv)^{p^a} = s^{p^a}u^{p^a}t^{p^a}v^{p^a} = s^{p^a}t^{p^a} \cdot 
  u^{p^a}v^{p^a}.$$
It follows that $s^{p^a}t^{p^a}$ is the semisimple part of $z$, whence 
$s^{p^a}t^{p^a} = z$. In particular, $(\la_1\mu_1)^{p^a} = \om$.
Hence, if $\la$ is the $p$-part of $\la_1$ and $\mu$ is the $p$-part of
$\mu_1$, then $(\la\mu)^{p^a} = \om$ (as $\om$ is a $p$-element).

\smallskip
3) Recall that $s$ is contained in a maximal torus of $\GL^{\eps}_n(q)$ which 
is a direct product $\prod^{m}_{i=1}C_i$, where $C_i$ is a cyclic group of 
order $q^{k_i}-\eps^{k_i}$ and $\sum^m_{i=1}k_i = n$. Then there is some $k$
with $1 \leq k \leq n$ such that $|\la|$ divides $(q^k-\eps^k)_p$.
Denote $c := \lfloor \log_pn \rfloor \geq 1$, and $p^{b_1} := (q-\eps)_p$ for 
$p > 2$ and $2^{b_1+1} := (q^2-1)_2$ for $p = 2$. 
Observe that the $p$-part of $k$ is at most $p^c$. Since $p$ divides 
$q-\eps$, $(q^k-\eps^k)_p$ is at most $(q^{p^c}-\eps^{p^c})_p$, which is 
$p^{c+b_1}$ by our choice of $b_1$.      

\smallskip   
4) We have shown that $\la^{p^{c+b_1}} = 1$, and 
similarly $\mu^{p^{c+b_1}} = 1$, whence $(\la\mu)^{p^{c+b_1}} = 1$. 
Assume we are in the case of (i). Then $(\la\mu)^{p^a} = \om$ and $|\om| = p^b$ by 
the result of 2). It follows
that $c+b_1 \geq a+b$. On the other hand,
$a+b = c+b_1+1$ by our choice of $a$, a contradiction.  

Assume we are in the case of (ii). Then $(\la\mu)^{p^a} = \om$ and $|\om| = p$ by 
the result of 2). It follows
that $c+b_1 \geq a+1$. On the other hand,
$a = c+b_1$ by our choice of $a$, again a contradiction.

In the case of (iii) we have $a=b=c=b_1 = 1$. Since $\la^{p^2} = \mu^{p^2} = 1$ and
$(\la\mu)^p = \om$ has order $p$, we may assume that $|\la| = p^2$. Recall that
$\la$ is the $p$-part of $\la_1$ and $(q-\eps)_p = p$. Hence
$\la_1 \in \F_{q^{2p}} \setminus \F_{q^2}$, and furthermore
$(\la_1)^{(q\eps)^i} = (\la_1)^{(q\eps)^j}$ precisely when $p|(i-j)$. The condition 
$x \in \SL^{\eps}_p(q)$ now implies that the $p$ eigenvalues $\la_1, \ldots ,\la_p$ 
of $x$ are exactly $\la_1^{(q\eps)^i}$, $0 \leq i < p$. In this case,
$$1 = \det(x) = \la_1^{\frac{(q\eps)^p-1}{q\eps-1}},$$
and so $ \la^{\frac{(q\eps)^p-1}{q\eps-1}} = 1$, a contradiction since 
$(\frac{(q\eps)^p-1}{q\eps-1})_p = p$. 
\end{proof}

As a consequence, we can now show that width $2$ need not hold for all finite 
quasisimple groups of large enough order, and similarly Waring's problem
need not have a positive solution for all finite 
quasisimple groups of large enough order.

\begin{cor}\label{4p}
Let $n > 2$ be any integer. Then there are finite quasisimple groups $G$ of 
arbitrarily large order such that $x^ny^n$ is not surjective on $G$.  
\end{cor}

\begin{proof}
Suppose that $n$ is divisible by an odd prime $p$. According to Dirichlet's Theorem,
there are infinitely many primes $q \equiv p+1 (\mod p^2)$. Now by Theorem
\ref{main2}(iii), $x^ny^n$ is not surjective on $\SL_p(q)$. 
In the remaining case,
$4|n$. Then again by Dirichlet's Theorem, there are infinitely many primes 
$q \equiv 3 (\mod 8)$, and for each such $q$, $x^ny^n$ is not surjective on
$\SL_2(q)$ by Lemma \ref{sl2}.
\end{proof}

The next statement shows that the obstruction for Waring's problem to have a 
positive solution is the center of the quasisimple group in question.

\begin{cor}\label{center}
Let $G$ be a finite quasisimple group and let $p$ be a prime divisor of $|Z(G)|$.
Then there is a power $p^a$ of $p$ (depending on $G$) such that 
$x^{p^a}y^{p^a}$ is not surjective on $G$.
\end{cor}

\begin{proof}
Modding out by $O_{p'}(Z(G))$, we may assume that $Z(G)$ is a $p$-group. Modding out
further by a suitable subgroup of $Z(G)$, we may assume that 
$Z(G) = \langle z \rangle$ has order $p$. Let $\al$ be an irreducible faithful
character of $Z(G)$ and let $\chi \in \Irr(G)$ lie above $\al$. Fix a 
prime $r \neq p$ and let $\varphi \in \IBr_r(G)$ be an irreducible constituent of
the reduction modulo $r$ of $\chi$. Then $\varphi$ is afforded by an absolutely 
irreducible representation $\Phi~:~G \to \GL_n(q)$ for some power $q$ of $r$.
Replacing $q$ by $q^{p-1}$ if necessary, we may assume that $p|(q-1)$. Since
$G$ is perfect, $\Phi(G) \leq \SL_n(q)$. But 
$\Phi(z) = \al(z)I_n$ and $\al(z)$ has order $p$. Hence 
$1 = \det(\Phi(z)) = \al(z)^n$ implies that $p|n$. Thus $p|\gcd(n,q-1)$, and
$\Phi(z)$ is a central element of order $p$ of $\GL_n(q)$. By 
Theorem \ref{main2}(ii), there is some $a$ (depending on $n$, $q$) such that 
$\Phi(z) \neq x^{p^a}y^{p^a}$ for all $x,y \in \GL_n(q)$. It follows that 
$z \neq x^{p^a}y^{p^a}$ for all $x,y \in G$.  
\end{proof}

The results in Lemma \ref{sl2}, Theorem \ref{main2}, and Corollary \ref{4p}
assume the groups in question have bounded rank. This constraint is removed in
the next two results.  

\begin{thm}\label{main3}
Let $p$ be a prime, $n \equiv p (\mod p^2)$, $\eps = \pm 1$,
$q$ a prime power with $q \equiv p+\eps (\mod p^2)$ if $p > 2$ 
and $q \equiv 4+\eps (\mod 8)$ if $p=2$. Also let $a = 2$ if $p > 2$ and 
$a = 3$ if $p = 2$. Then $x^{p^a}y^{p^a}$ is not surjective on $G = \SL^{\eps}_n(q)$.
\end{thm}

\begin{proof}
1) Our assumptions imply that $(\gcd(n,q-\eps))_p = p$, hence 
$O_p(Z(G)) = \langle z \rangle$ has order $p$, and $z = \om I_n$ with $|\om| = p$.
Assume the contrary: $z = x^ry^r$ for some $x,y \in G$ and $r := p^a$. Consider 
the Jordan decompositions $x = su$ and $y = tv$, with $s,t$ semisimple. 
As in part 2) of the proof of Theorem \ref{main2}, we see that $z=s^rt^r$ and 
$s^r$, $t^r$ commute. Let $s = s_1s_2$ and $t = t_1t_2$, where $s_1$, resp. $t_1$,
is the $p$-part of $s$, resp. of $t$, and $s_2$, resp. $t_2$,
is the $p'$-part of $s$, resp. of $t$. Then again $s_1^r$, $s_2^r$, $t_1^r$, $t_2^r$
all commute with each other, and $s_1^rt_1^r$ is a $p$-element and $s_2^rt_2^r$ is 
a $p'$-element. Now 
$$z = (s_1s_2)^r(t_1t_2)^r = s_1^rs_2^rt_1^rt_2^r = s_1^rt_1^r \cdot 
  s_2^rt_2^r.$$
It follows that $z = s_1^rt_1^r$. Without loss we may replace $s$ by $s_1$, $t$ by 
$t_1$, and assume that $s$ and $t$ are $p$-elements with $z = s^rt^r$.

Suppose $\la_1, \ldots ,\la_n$ are all the eigenvalues of the matrix $s$
and $\mu_1, \ldots ,\mu_n$ are all the eigenvalues of the matrix $t$
(with counting multiplicities). Arguing as in part 2) of the proof of Theorem 
\ref{main2}, we see that $(\la_1\mu_1)^r, \ldots ,(\la_n\mu_n)^r$ are all the 
eigenvalues of $s^rt^r = z$ (with counting multiplicities). More precisely, if
$V$ is the natural module for $\SL^{\eps}_n(q)$ over $\bar{\F}_q$, then 
there is a decomposition $V = \oplus^n_{i=1}V_i$ such that 
$s^r$ acts on $V_i$ as the scalar multiplication by $\la_i^r$, and 
$t^r$ acts on $V_i$ as the scalar multiplication by $\mu_i^r$. It follows that
for all $i$, $(\la_i\mu_i)^r = \om$ and so $|\la_i\mu_i| = pr$.

\smallskip
2) As in part 3) of the proof of Theorem \ref{main2}, for each $i$ there is a 
smallest integer $k_i$ between $1$ and $n$ such that $\la_i^{q^{k_i}-\eps^{k_i}} = 1$.
Since $p|(q-\eps)$, the minimality of $k_i$ implies that $k_i$ is a $p$-power.
We claim that either

(i) $|\la_i| = rk_i/p$, or

(ii) $k_i = 1$, and $|\la_i|$ divides $r/p^2$.

First consider the case $p>2$. Since $p|(q-\eps)$ and $k_i$ is a $p$-power, we have 
that $(q^{k_i}-\eps^{k_i})_p = k_i \cdot (q-\eps)_p = pk_i$, and so $|\la_i|$ divides
$pk_i = rk_i/p$. Suppose that (i) does not hold. Then $|\la_i|$ divides 
$k_i$. If $k_i = 1$, then $|\la_i|=1=r/p^2$, and so (ii) holds.  
If $k_i > 1$, then $p|k_i$. In this 
case, $(q^{k_i/p}-\eps^{k_i/p})_p = (k_i/p) \cdot (q-\eps)_p = k_i$, whence 
$\la_i^{q^{k_i/p}-\eps^{k_i/p}} = 1$, contrary to the choice of $k_i$.  

Next assume that $p=2$. Since $4|(q-\eps)$ and $k_i$ is a $2$-power, we have that
$(q^{k_i}-\eps^{k_i})_2 = k_i \cdot (q-\eps)_2 = 4k_i$ and so $|\la_i|$ divides
$4k_i = rk_i/2$. Suppose that (i) does not hold. Then $|\la_i|$ divides 
$2k_i$. If $k_i = 1$, then $|\la_i|$ divides $2 = r/4$, and so (ii) holds.  
If $k_i \geq 2$, then $2|k_i$. In this 
case, $(q^{k_i/2}-\eps^{k_i/2})_2 = (k_i/2) \cdot (q-\eps)_2 = 2k_i$, whence 
$\la_i^{q^{k_i/2}-\eps^{k_i/2}} = 1$, contrary to the choice of $k_i$, and the claim
follows.

Similarly, for each $i$ there is a smallest integer $\ell_i$ between $1$ and $n$ 
such that $\mu_i^{q^{\ell_i}-\eps^{\ell_i}} = 1$. The same argument as above shows that 
$\ell_i$ is a $p$-power; furthermore, 
either $|\mu_i| = r\ell_i/p$, or $\ell_i = 1$ and $|\mu_i|$ divides $r/p^2$.

\smallskip
3) Consider the map $J$ on $\bar{\F}_q$ defined via $J(\al) = \al^{q\eps}$. Then
the $J$-orbit of $\la_i$ has length $k_i$, and the $J$-orbit of $\mu_i$ has 
length $\ell_i$. Next we define:
$$X = \{i \mid 1 \leq i \leq n,~p^2|k_i,~p^2|\ell_i\},~~~   
  Y = \{i \mid 1 \leq i \leq n,~p^2|k_i,~p^2\not{|}\ell_i\},$$
$$Z = \{i \mid 1 \leq i \leq n,~p^2\not{|}k_i,~p^2|\ell_i\},~~~   
  T = \{i \mid 1 \leq i \leq n,~p^2\not{|}k_i,~p^2\not{|}\ell_i\}.$$
Here we show that $p^2$ divides $|X|$, $|Y|$, and $|Z|$.

3a) Suppose for instance that $1 \in X$. 
Then $p^2|k_1$, and so by the results of 2),
$|\la_1| = rk_1/p$. Similarly, $|\mu_1| = r\ell_1/p$. Since $s \in \SL^{\eps}_n(q)$,
$J(\la_1)$ is an eigenvalue of $s$, whence $J(\la_1) = \la_i$ for some $i$ between
$1$ and $n$. We now show that {\it every $i$ with $\la_i = J(\la_1)$ must 
belong to $X$}. Indeed, 
$|\la_i| = |\la_1^{q\eps}| = |\la_1| = rk_1/p$ is divisible by $p^3$. So by 
2), $|\la_i| = rk_i/p$, and $k_i = k_1$ is divisible by $p^2$. Next, since 
$|\om| = p$ divides $q-\eps$, we see that $\om^{q\eps} = \om$. Hence,
$$\la_i^r\mu_i^r = \om = \om^{q\eps} = ((\la_1\mu_1)^r)^{q\eps} = 
  (\la_1^{q\eps})^r(\mu_1)^{rq\eps} = \la_i^r (\mu_1)^{rq\eps},$$
and so $\mu_i^r = (\mu_1)^{rq\eps}$, whence $|\mu_i^r| = |\mu_1^r|$. On the other 
hand, $|\mu_1^r| = \ell_1/p$ is divisible by $p$. It follows that 
$|\mu_i^r| = \ell_1/p$ is divisible by $p$ and so $|\mu_i| = r\ell_1/p$ is divisible
by $p^3$. Again by 2), we now have that $|\mu_i| = r\ell_i/p$, and 
$\ell_i = \ell_1$ is divisible by $p^2$. Thus $i \in X$. 

Repeating this argument with $\la_1$ replaced by $J^b(\la_1)$ for any $b$, we 
see that every $i$ with $\la_i$ in the $J$-orbit of $\la_1$ must belong to 
$X$. Certainly, all these $\la_i$ occur with the same multiplicity in
the spectrum of $s$ on $V$.  We have shown that $J$ acts on the set 
$\{V_i \mid i \in X\}$ (according to the 
action of $J$ on the $\la_i$'s), and each orbit has length divisible by $p^2$. 
Therefore $p^2$ divides $|X|$.

3b) Similarly, now suppose for instance that $1 \in Y$. 
Then $p^2|k_1$, and so by 2), $|\la_1| = rk_1/p$. As above, 
$J(\la_1) = \la_i$ for some $i$ between $1$ and $n$. We now show that {\it every 
$i$ with $\la_i = J(\la_1)$ must belong to $Y$}. Indeed, 
$|\la_i| = |\la_1^{q\eps}| = |\la_1| = rk_1/p$ is divisible by $p^3$. So by 
2), $|\la_i| = rk_i/p$, and $k_i = k_1$ is divisible by $p^2$. Next, as above we 
also have $\mu_i^r = (\mu_1)^{rq\eps}$, whence $|\mu_i^r| = |\mu_1^r|$. Since 
$1 \in Y$, $\ell_1 = p$ or $\ell_1 = 1$. By the results of 2), $\mu_1^r = 1$, 
and so $\mu_i^r = 1$. This in turn implies by 2) that $p^2 {\!\not{|}} \ell_i$,
i.e. $i \in Y$. 

Repeating this argument with $\la_1$ replaced by $J^b(\la_1)$ for any $b$, we 
see that every $i$ with $\la_i$ in the $J$-orbit of $\la_1$ must belong to 
$Y$. As before, all these $\la_i$ occur with the same multiplicity in
the spectrum of $s$ on $V$. We have shown that 
$J$ acts on the set $\{V_i \mid i \in Y\}$ (according to the action of $J$ on
the $\la_i$'s), and each orbit has length divisible by $p^2$, hence $p^2$ divides 
$|Y|$.

Reversing the roles of $\la_i$'s and $\mu_i$'s, we see that $p^2$ divides $|Z|$. 

\smallskip
4) We have shown that $p^2$ divides $|X|$, $|Y|$, and $|Z|$. But 
$|X|+|Y|+|Z|+|T| = n \equiv p (\mod p^2)$, hence $T \neq \emptyset$. Thus 
we can find $m$ such that $p^2 {\!\not{|}} k_m$ and 
$p^2 {\!\not{|}} \ell_m$. By the 
results of 2), this implies that $\la_m^r = 1 = \mu_m^r$. But in this case,
$\om = (\la_m\mu_m)^r = 1$, a contradiction.   
\end{proof}

\begin{thm}\label{main4}
{\rm (i)} Suppose $n \equiv 3 (\mod 9)$, $\eps = \pm 1$, and 
$q \equiv 3+\eps (\mod 9)$. 
Then $x^9y^9$ is not surjective on $\SL^{\eps}_{n}(q)$.

{\rm (ii)} Suppose $n \equiv 1 (\mod 2)$ and $q \equiv 5 (\mod 8)$. 
Then $x^8y^8$ is not surjective on $\Sp_{2n}(q)$.

{\rm (iii)} Suppose $n \equiv 1 (\mod 2)$, 
$\eps = \pm$ and $q \equiv 4 +\eps 1 (\mod 8)$. 
Then $x^8y^8$ is not surjective on $\Omega^{\eps}_{2n}(q)$ and 
$\Spin^{\eps}_{2n}(q)$.

{\rm (iv)} Suppose $q \equiv \pm 3 (\mod 8)$ and $n \geq 2$. 
Then $x^{2^{n+2}}y^{2^{n+2}}$ is not surjective on $\Spin_{2n+1}(q)$.

{\rm (v)} Suppose $\eps = \pm 1$ and $q \equiv 3+\eps (\mod 9)$. 
Then $x^{81}y^{81}$ is not surjective on the simply connected group 
$E_6^{\eps}(q)$.

{\rm (vi)} Suppose $q \equiv \pm 3 (\mod 8)$. Then $x^{128}y^{128}$ is not 
surjective on the simply connected group $E_{7}(q)$.
\end{thm}

\begin{proof}
(i) This is the particular case $p = 3$ of Theorem \ref{main3}.

(ii) Embed $\Sp_{2n}(q)$ in $G = \SL_{2n}(q)$. Then the central involution 
$z$ of $\Sp_{2n}(q)$ is the unique generator of $O_2(Z(G))$. 
Now the proof of Theorem \ref{main3} shows that
$z \neq x^8y^8$ for all $x,y \in G$.  

(iii) The conditions on $(n,q,\eps)$ imply that 
$Z(\Omega^{\eps}_{2n}(q)) = \langle z \rangle \cong C_2$ and 
$Z(\Spin^{\eps}_{2n}(q)) \cong C_4$. Embed $\Omega^{\eps}_{2n}(q)$ in 
$G = \SL^{\eps}_{2n}(q)$. Then the central involution 
$z$ of $\Omega^{\eps}_{2n}(q)$ is the unique generator of $O_2(Z(G))$. 
Now we can argue as in (ii).

(iv) The {\it spin} representation embeds $\Spin_{2n+1}(q)$ (irreducibly)
into $G = \SL_{2^n}(q)$, so that $Z(G)$ contains the central involution
$z$ of $\Spin_{2n+1}(q)$, cf. \cite[Proposition 5.4.9]{KL}. 
Now we can apply Theorem \ref{main2}(ii). 

(v) One can embed $E^{\eps}_6(q)$ (irreducibly) into $G = \SL^{\eps}_{27}(q)$, 
so that $Z(G)$ contains a central element $t$ of $E^{\eps}_6(q)$ of order
$3$, cf. \cite[Proposition 5.4.17]{KL}. Now apply Theorem \ref{main2}(ii).

(vi) One can embed $E_7(q)$ (irreducibly) into $G = \SL_{56}(q)$, so that $Z(G)$ 
contains the central involution $z$ of $E_7(q)$, cf. 
\cite[Proposition 5.4.18]{KL}. Now apply Theorem \ref{main2}(ii).
\end{proof}

\bigskip

\section{Products of squares in finite quasisimple groups}
Corollary \ref{4p} leaves out the word $x^2y^2$. It was shown in 
\cite{LBST2} (and independently in \cite{GM}) that $x^2y^2$ is surjective 
on any 
non-abelian simple group. It turns out that this word
is also surjective on any finite quasisimple group:

\begin{thm}\label{main5}
Let $G$ be any finite quasisimple group. Then $x^2y^2$ is surjective on $G$.
\end{thm}

\begin{proof}
1) Without loss we may assume that $Z(G)$ is a nontrivial $2$-group. (Indeed,
if $Z_1 = O_{2'}(Z(L))$, then the center of the quasisimple group $G/Z_1$ is
a $2$-group. Applying the main result of \cite{LBST2} in the case 
$Z(G/Z_1) = 1$ and our hypothesis otherwise, we see that for any $g \in G$, there 
are $x,y \in G$ and $t \in Z_1$ such that $g = x^2y^2t$. Since $|Z_1|$ is odd,
$t = z^2$ for some $z \in Z_1$, and so $g = (xz)^2y^2$.) So in what follows
we will consider only such finite quasisimple groups $G$, and let $S := G/Z(G)$.

Recall that $g \in G$ is a product of two squares if and only if 
$$(\star)~:~\sum_{\chi~ = ~\bar{\chi}~ \in \Irr(G)}\frac{\chi(g)}{\chi(1)} \neq 0.$$
Now let $\cS$ be the collection of the following simple groups:
$\AAA_n$, $5 \leq n \leq 13$, $\Sp_6(2)$, $G_2(4)$, $F_4(2)$, 
$\PSL_3(4)$, $\PSU_6(2)$, $^2E_6(2)$, $\Omega^+_8(2)$, $^2B_2(8)$, 
$M_{12}$, $M_{22}$, $J_2$, $HS$, $Suz$, $Ru$, $Co_1$, $Fi_{22}$, $BM$.
Using the criterion $(\star)$ and character tables available in {\sf GAP} 
and {\sf Magma}, Eamonn O'Brien has checked the statement in the cases
$S \in \cS$. Thus we may assume $S \notin \cS$. Then either 

(i) $G = 2\AAA_n$ with $n \geq 14$, or 

(ii) $G$ is a quotient of a quasisimple group $L$ of Lie type in 
{\it odd} characteristic $p$ of simply connected type.

\smallskip
2) We will handle the case (i) of alternating groups by induction 
on $n$. Observe that any 
element in $2\AAA_4 \cong \SL_2(3)$ is a product of two squares. Hence 
the induction base $4 \leq n \leq 13$ has been established. Following 
the approach of \cite{LBST2}, we say that $g \in 2\AAA_n$ is 
{\it breakable}, if $g = xy$ lies in a central product $2\AAA_r * 2\AAA_{n-r}$
(so $x \in 2\AAA_r$ and $y \in 2\AAA_{n-r}$) with $4 \leq r \leq n-4$. 
Applying the induction hypothesis to $r$ and $n-r$, 
we see that any breakable element 
in $2\AAA_n$ is a product of two squares. Assume $g \in 2\AAA_n$ is 
unbreakable. The proof of Lemma 3.3 of \cite{LBST2} shows that $g$ 
must satisfy 
one of the conclusions of this Lemma; in particular,
$|C_{\AAA_n}(g)| \leq (3/4) \cdot (n-3)^2$. Now Lemmas 3.4 and 3.6 of 
\cite{LBST2}
imply that 
$$\sum_{1_G \neq \chi \in \Irr(G)} \frac{|\chi(g)|}{|\chi(1)|} < 0.876.$$ 
Hence $g$ satisfies $(\star)$ and so $g$ is a product of two squares. Thus
the induction step has been completed.

\smallskip
3) From now on we may assume that we are in the case (ii). By \cite[\S4]{EG},
every non-central element $g \in L$ is a product of two unipotent elements
$g = uv$. Then $u$ and $v$ are of odd order and so they are squares in $L$. 
It remains therefore to show that any non-trivial $2$-element in $Z(L)$ is 
a product of two squares. 

Here we consider the case $L = \SL^{\eps}_{n}(q)$ and 
assume that $z = \om I_n \in Z(L)$ has order $2^a \geq 2$. Denote 
$(q-\eps)_2 = 2^b$, so that $a \leq b$ and $2^a|n$. Embedding $z$ into
a direct product 
$$\SL^{\eps}_{2^a}(q) \times \SL^{\eps}_{2^a}(q) \times \ldots \times 
  \SL^{\eps}_{2^a}(q) < L,$$
we may furthermore assume that $n = 2^a$. 

Suppose first that $b \geq a+1$. Then the diagonal torus $C_{q-\eps}^n$ of 
$\GL^{\eps}_{n}(q)$ contains the element 
$$x = \diag(\al, \al, \ldots, \al, -\al)$$ 
with $\al^2 = \om$. Now $\det(x) = -\al^{2^a} = -\om^{2^{a-1}} = 1$, i.e.
$x \in L$, and $x^2 = z$.

Now suppose that $b = a = 1$; in particular, $L \cong SL_2(q)$. Then the cyclic 
torus $C_{q+\eps}$ contains an element $y$ of order $4$, and $y^2 = z$.

We may now assume that $b = a \geq 2$. Since $n = 2^a \geq 4$, we have 
$$(q^{n/2}-\eps^{n/2})_2 = (q^2-1)_2 \cdot \frac{n}{4} = 2^{2a-1}.$$
Choose $\gam \in \bar{\F}_q$ of order $2^{2a-1}$ such that $\om = \gam^{2^{a-1}}$.
Now the cyclic torus $C_{q^{n/2}-\eps^{n/2}}$ of $\GL^{\eps}_{n/2}(q)$ contains 
an element $t$ of order $2^{2a-1}$ (conjugate to 
$$\diag(\gam,\gam^{q\eps}, \ldots, \gam^{(q\eps)^{n/2-1}})$$
over $\bar{\F}_q$). Set 
$$x := \diag(t^{2^{a-2}}, t^{3 \cdot 2^{a-2}}),~~~
  y = \diag(I_{n/2-1},-1,\om^{-1}I_{n/2}).$$
Then $x \in \GL^{\eps}_n(q)$ and 
$$\det(x) = (\det(t))^{2^a} = \gam^{2^a \cdot \frac{(q\eps)^{n/2}-1}{q\eps-1}} = 1,$$
i.e. $x \in L$. By its choice, $y$ lies in the diagonal torus $C_{q-\eps}^n$ of 
$\GL^{\eps}_n(q)$ and $\det(y) = -\om^{-n/2} = 1$, i.e. $y \in L$. 
It remains to observe that $z = x^2y^2$.

In particular, we have shown that any element in $\SL_2(q)$ with $q \geq 3$ odd 
is a product of two squares.

\smallskip
4) Now we will consider the remaining Lie-type groups $L$ of odd characteristic.
If $L = \Sp_{2n}(q)$, then we can embed the central involution $z$ of $L$ into
a direct product $\Sp_2(q) \times \ldots \times \Sp_2(q) < L$ and apply the 
result of 3) to each factor $\Sp_2(q) \cong \SL_2(q)$. 

Suppose $L = \Spin_n(q)$ with $n \geq 7$ odd, or $L = \Spin^{\eps}_n(q)$ with 
$n \geq 8$ even and $q^{n/2}- \eps \equiv 2 (\mod 4)$. Then the central involution
$z$ of $L$ can be embedded in $\Spin_3(q) \cong \SL_2(q)$ (see 
\cite[Lemma 4.1]{LBST2}), and so we can apply the results of 3) 
to $\SL_2(q)$.  

Suppose $L = \Spin^{\eps}_{2n}(q)$ with $n \geq 4$ and $4|(q^n - \eps)$. 
Assume in addition that $2|n$ (and so $\eps = +$). Then embed 
$Z(\Omega^+_{2n}(q)) \cong C_2$ in a direct product
$$\Omega^+_{4}(q) \times \ldots \times \Omega^+_4(q) < \Omega^+_{2n}(q).$$ 
By \cite[Lemma 4.1]{LBST2}, this embeds $Z(L)$ (of order $4$) in the central
product  
$$\Spin^+_{4}(q) * \ldots * \Spin^+_4(q) < L.$$     
Since $\Spin^+_4(q) \cong \SL_2(q) \times \SL_2(q)$, we can now apply the 
results of 3) to each factor $\Spin^+_4(q)$. Now assume that $n$ is odd, 
whence $n \geq 5$ and $4|(q-\eps)$. Then embed 
$Z(\Omega^{\eps}_{2n}(q)) \cong C_2$ in a direct product
$$\Omega^+_{4}(q) \times \ldots \times \Omega^+_4(q) \times \Omega^{\eps}_6(q) < 
  \Omega^{\eps}_{2n}(q).$$ 
By \cite[Lemma 4.1]{LBST2}, this embeds $Z(L)$ (of order $4$) in the central
product  
$$\Spin^+_{4}(q) * \ldots * \Spin^+_4(q) * \Spin^{\eps}_6(q) < L.$$     
Since $\Spin^+_4(q) \cong \SL_2(q) \times \SL_2(q)$ and 
$\Spin^{\eps}_6(q) \cong \SL^{\eps}_4(q)$, we can now apply the 
results of 3) to each factor of the last subgroup.

It remains to consider the central involution $z$ of the simply connected group
$L$ of type $E_7(q)$. By \cite[Lemma 5.1]{LBST2}, we can embed $Z(L)$ into 
$\Spin^+_{12}(q)$, and so we are done by the previous step. 
\end{proof}

\end{document}